\documentclass[preprint]{imsart}

\RequirePackage[OT1]{fontenc}
\RequirePackage{amsthm,amsmath}
\usepackage{amssymb, enumerate}
\RequirePackage[numbers]{natbib}
\RequirePackage[colorlinks,citecolor=blue,urlcolor=blue]{hyperref}

\arxiv{arXiv:0000.0000}

\startlocaldefs
\numberwithin{equation}{section}
\theoremstyle{plain}
\newtheorem{thm}{Theorem}[section]
\newtheorem{lemma}{Lemma}[section]

\endlocaldefs

\def\rank{{\rm rank}}

\def\trace{\mathop{\rm tr}}

\begin{document}

\begin{frontmatter}
\title{Wald Statistics in high-dimensional PCA}

%\runtitle{Statistical Inference in high-dimensional PCA}
%\thankstext{T1}{Footnote to the title with the ``thankstext'' command.}
	\runtitle{Wald Statistics in high-dimensional PCA}
\begin{aug}
%\author{\fnms{First} \snm{Author}\thanksref{t1,t2,m1}\ead[label=e1]{first@somewhere.com}},
%\author{\fnms{Second} \snm{Author}\thanksref{t3,m1,m2}\ead[label=e2]{second@somewhere.com}}
%\and
%\author{\fnms{Third} \snm{Author}\thanksref{t1,m2}
%\ead[label=e3]{third@somewhere.com}
%\ead[label=u1,url]{http://www.foo.com}}
%
%\thankstext{t1}{Some comment}
%\thankstext{t2}{First supporter of the project}
%\thankstext{t3}{Second supporter of the project}
%\runauthor{F. Author et al.}
%
%\affiliation{Some University\thanksmark{m1} and Another University\thanksmark{m2}}
%
%\address{Address of the First and Second authors\\
%Usually a few lines long\\
%\printead{e1}\\
%\phantom{E-mail:\ }\printead*{e2}}
%
%\address{Address of the Third author\\
%Usually a few lines long\\
%Usually a few lines long\\
%\printead{e3}\\
%\printead{u1}}

\author{\fnms{Matthias} \snm{L\"offler}}%,e3},
	\ead[label=e1]{m.loffler@statslab.cam.ac.uk}
%\and
%\author{\fnms{Richard} \snm{Nickl}\thanksref{d,e4}\ead[label=e4, mark]{r.nickl@statslab.cam.ac.uk}}

\address{Statistical Laboratory, 
	Centre for Mathematical Sciences, University of Cambridge, Wilberforce Road,
	CB3 0WB Cambridge,
	United Kingdom.
	\printead{e1}}
\end{aug}

\begin{abstract}
In this note we consider PCA for Gaussian observations $X_1,\dots, X_n$ with covariance $\Sigma=\sum_i \lambda_i P_i$ in the 'effective rank' setting with model complexity governed by $\mathbf{r}(\Sigma):=\trace(\Sigma)/\| \Sigma \|$.  \\
We prove a Berry-Essen type bound for a  Wald Statistic of the spectral projector $\hat P_r$. This can be used to construct non-asymptotic confidence ellipsoids and tests for spectral projectors $P_r$. Using higher order pertubation theory we are able to show that our Theorem remains valid even when $\mathbf{r}(\Sigma) \gg \sqrt{n}$. 
%Particularly, our proofs show how to adjust 'classic' techniques which often only work in the regime $r(\Sigma):=\trace(\Sigma)/\| \Sigma \|=o(n^{1/2})$ to the high-dimensional setting $ n \gg r(\Sigma) \gg n^{1/2}$. 
%If $r(\Sigma)=o(n^{1/2})$ one can expect that standard techniques from semi-parametric statistics work well. For the more challenging regime $n \gg r(\Sigma) \gg n^{1/2}$ we show how to adjust those techniques. 
%Building upon the work of \cite{KoltchinskiiLouniciPCAAHP} and \cite{KoltchinskiiLouniciPCAAOS} we consider statistical inference and efficiency for spectral projectors $P_r$ and eigenvectors $\theta_r$. +
%For the proof of our results we calculate the Fisher information for $P_r$. We prove semi-parametric efficiency for de-biased estimators of spectral projectors and eigenvectors. Moreover, we prove a Berry-Essen type bound for a normalised Wald statistic of $\hat P_r-P_r$. This enables the construction of tests and confidence sets for $P_r$. Finally, using van Tree's inequality we show that our estimators for $\langle P_r u, v \rangle$ and $\langle \theta_r, u \rangle$ are minimax optimal,  exactly matching our lower bound. 

\end{abstract}

%\begin{keyword}[class=MSC]
%\kwd[Primary ]{60K35}
%\kwd{60K35}
%\kwd[; secondary ]{60K35}
%\end{keyword}


\end{frontmatter}

\section{Introduction}
Principal Component Analysis (PCA) is a widely used dimension-reduction technique in statistics. In the traditional Gaussian setting going back to Anderson \cite{Anderson1963} one observes $n$ $i.i.d.$ zero mean Gaussian random variables with covariance matrix $\Sigma$. In more recent years mainly three working assumptions on $\Sigma$ have been considered: The \emph{spiked covariance model}, the \emph{spiked sparse covariance model} and the \emph{'effective rank'} setting.  \\
Johnstone \cite{JohnstoneAos2001} introduced the spiked covariance model in which $\Sigma$ is given by
\begin{equation} \label{Spiked covariance model}
\Sigma = \sum_{i=1}^l s_i \theta_i \theta_i^T + \sigma^2 I_p.
\end{equation}
Subsequent work in this model has mainly been focused on the failures of PCA in high-dimensions when the dimension $p \rightarrow \infty$ and $p/n \rightarrow const.$ \cite{JohnstoneLu, PaulStatisticaSinica, Nadler2008, WangFan}. \\
A remedy is to assume that the leading eigenvectors $\theta_i$ in \eqref{Spiked covariance model} are sparse, enabling thus inference even when $p/n \rightarrow \infty$ \cite{CaiMaWu2013, GaoZhou, VuLei2013, BerthetRiggolet, WangBerthetSamworth}. \\
We will consider the effective rank setting \cite{ReissWahl, KoltchinskiiLofflerNickl, KoltchinskiiLouniciPCAAHP, KoltchinskiiLouniciPCAAOS, KoltchinskiiLouniciPCABernoulli, NaumovSpokoinyUlyanov}. Here no assumptions on the particular structure of $\Sigma$ are made, except that the effective rank  $\mathbf{r}(\Sigma):=\trace(\Sigma)/\| \Sigma \|=o(n)$ where  $\trace(\cdot)$ denotes the trace and $\| \cdot \|$ the operator norm. This allows for a wider range of models, for example $\Sigma$ with a polynomial or exponential decay of the eigenvalues \cite{ReissWahl}. \\
Rates of convergence and limiting results for empirical spectral projectors and empirical eigenvectors were proven in \cite{KoltchinskiiLouniciPCAAHP}. Building upon this \cite{KoltchinskiiLofflerNickl} proved that a de-biased empirical eigenvector attains the semi-parametric efficiency bound. \\
A method for constructing Frobenius type confidence sets for $ P_r$ was given in the two papers \cite{KoltchinskiiLouniciPCAAOS} and \cite{KoltchinskiiLouniciPCAarxiv}. However, their approach requires sample splitting into three samples and the assumption that $\| \Sigma \|_F^2 \rightarrow \infty$. The latter is not necessarily fulfilled, for example in a decaying eigenvector model where the eigenvalues $\lambda_i$ scale as  $i^{-\alpha}$, $\alpha > 1/2$. Others approaches based on the bootstrap and Bayesian inference, respectively, were proposed by \cite{NaumovSpokoinyUlyanov} and \cite{SilinSpokoiny} but require at least  $\mathbf{r}(\Sigma)=o(n^{1/3})$ and do not deal with the harder case $\mathbf{r}(\Sigma) \gg \sqrt{n}$ where one needs to account for bias. \\ \\
In this note we present a third possibility based on Wald Statistics for constructing a Frobenius type confidence set for $P_r$. We show that when $p \rightarrow \infty$ a further normalized Wald statistic of $\hat P_r-P_r$ is asymptotically Gaussian. \\ % which does not require sample splitting and still works in the high-dimensional regime $\mathbf{r}(\Sigma) \gg n$. \\ \\
Our main contribution is that we show how to deal with the critical case $\mathbf{r}(\Sigma) \gg \sqrt{n}$ by using second order pertubation theory, requiring for example in the spiked covariance model \eqref{Spiked covariance model} that $\mathbf{r}(\Sigma) =o(n^{2/3})$. 
	
	% Contrary to the construction in \cite{KoltchinskiiLouniciPCAarxiv} our approach does not require sample splitting or that $\| \Sigma \|_F^2 \rightarrow \infty$.  Instead we need to assume that the smallest eigenvalue of $\Sigma$ %, $\lambda_{\min}$, 
	%is bounded from below. %This resembles  $\beta_{\min}$-conditions on the minimal non-zero entry of the regressor in sparse regression (e.g. \cite{vdgBuhlmannZhou}) and is required for consistent estimation of the Fisher-information. 
\subsection{Set-up \& notation} \label{section prelim}
%We assume the same set-up and similar notation as Koltchinskii and Lounici in  \cite{KoltchinskiiLouniciPCAAHP, KoltchinskiiLouniciPCAAOS, KoltchinskiiLouniciPCAarxiv}. 
For matrices $A,B,C$ we define the Kronecker product as $(A \otimes B)C=ACB^T$ and the Frobenius inner product $\langle A, B \rangle:= \trace(A^TB).$  $\| \cdot \|_F$ denotes the Frobenius-norm.
%Moreover, for $A,B, C \in \mathbb{H} \otimes \mathbb{H}$ we define the Kronecker-product \begin{equation}
%(A \times B)C:=ACB^*.
%\end{equation} \label{Def: Kronecker product} 
The notation $\| \cdot \|$ will be used for the operator norm and in slight abuse of notation for the Euclidean norm of vectors with corresponding Euclidean inner product $\| v \|:=\langle v, v \rangle:=v^Tv$. \\% and $\| \cdot \|_*$ the trace-norm.  
We will frequently use the following convention throughout the paper: for nonnegative $a,b$ $a \lesssim b$ means that there exists a constant $C$ not depending on $n$ or $\textbf{r}(\Sigma)$
 such that $a \leq Cb$. \\  We assume that $X_1,...X_n$ are $i.i.d$ centred Gaussian vectors in $\mathbb{R}^p$ with $\mathbb{E} \| X \|^2 < \infty$. We denote by $\Sigma=\mathbb{E}X_1 X_1^T$ the covariance matrix of the observations $X_1, \dots, X_n$ and we denote the empirical covariance matrix by
\begin{equation}
\hat \Sigma= :\frac{1}{n} \sum_{j=1}^n X_j X_j^T.
\end{equation}
We define the \emph{effective rank}: $\textbf{r}(\Sigma):=\trace (\Sigma)/\|\Sigma \|.$ \\
Since $\Sigma$ is symmetric and positive semidefinite it has spectral decomposition $\Sigma=\sum_{s} \lambda_s P_s$ where $\lambda_s$ are distinct strictly positive, descending eigenvalues and $P_s$ are the corresponding spectral projectors. Let $\mu_j$ denote the eigenvalues of $\Sigma$ arranged in a non-increasing order and repeated with their multiplicities. Define $\Delta_r:=\{j: \mu_j =\lambda_r\}$. 
As described (and proven in Lemma 2.2.) in \cite{KoltchinskiiLofflerNickl} it suffices to assume that $\Delta_r$ is known as $\Delta_r$ can be identified on an event of high probability under our assumptions. \\
We thus denote by $\hat P_r$ the spectral projector corresponding to the eigenvalues $\{ \mu_j(\hat \Sigma), j \in \Delta_r\}$ and by $\hat \lambda_r$ one arbitrary chosen eigenvalue from the same set.  %If $\| \hat \Sigma - \Sigma \| < \frac{\bar g_r}{4}$, then the eigenvalues $\{ \mu_j(\hat \Sigma), j \in \Delta_r \}$ lie in an interval of radius $\bar g_r/4$ around $\lambda_r$ and all other eigenvalues of $\hat \Sigma$ are at least $3 \bar g_r/4$ from $\lambda_r$ away. Thus, for each $r$ there is a cluster of eigenvalues $\{ \mu_j(\hat \Sigma),~ j \in \Delta_r\}$ and the corresponding spectral projector $\hat P_r$ is a natural estimator for $P_r$ with $\hat m_r=m_r$. 
We denote by $\bar g_r:=\min(\lambda_{r-1}-\lambda_r, \lambda_r-\lambda_{r+1})$ the spectral gap of the eigenvalue $\lambda_r$ with the convention that $\lambda_0=\infty$.
By $\lambda_{\min}$ we denote the smallest eigenvalue of $\Sigma$ and likewise by $\hat {\lambda}_{\min}$ the smallest eigenvalue of $\hat \Sigma$. 
If $\text{card}(\Delta_r):=m_r=1$ we define $\hat \theta_r$ to be a sample eigenvector belonging to the eigenvalue $\hat \lambda_r$. \\ %We denote by $m_r$ the multiplicity of the eigenvalue $\lambda_r$. We denote by $\hat P_r$ the $r$'th spectral projector of the sample covariance matrix $\hat \Sigma$ and likewise, by $\hat \theta_r$ the corresponding sample eigenvector. $\hat \lambda_r$ denotes the $r$'th largest empirical eigenvalue of $\hat \Sigma$.  \\
\section{Main result}
Wald statistics \cite{Wald} are commonly used when the dimension of the parameter space is $p=const.$ %Using standard theory for the MLE and the delta method one obtains if the rank of $P_r$ is $1$ and $p=const.$ that
\noindent The Fisher information for the model $X\sim N(0,\Sigma)$ is ${\mathbb  I}(\Sigma)
=\frac{1}{2}(\Sigma^{-1}\otimes \Sigma^{-1})$ (e.g. \cite{Eaton}). If $p$ is constant the maximum likelihood estimator $\hat \Sigma$ for $n$ i.i.d. observations is asymptotically Gaussian distributed with $\sqrt{n}$-rate and limiting covariance 
${\mathbb I}(\Sigma)^{-1}=2(\Sigma \otimes \Sigma).$ 
Applying the delta method to  
$g(\Sigma):= P_r$ shows that $g(\hat \Sigma)$ is asymptotically Gaussian, too, and has limiting covariance 
\begin{equation} \label{Waldparametric}
 \mathbb{I}(P_r)^{-1}:=2 \sum_{s \neq r} \left ( P_s \otimes P_r + P_r \otimes   P_s  \right) \frac{\lambda_r \lambda_s}{(\lambda_r-\lambda_s)^2} = 2(\Sigma C_r^2 \otimes  P_r \Sigma+ \Sigma  P_r \otimes  C_r^2 \Sigma),
\end{equation}
where we define the resolvent operator $C_r:=\sum_{s \neq r} \frac{P_s}{\lambda_r-\lambda_s}$. \\
If $p$ remains fixed and $\rank(P_r)=m_r$ it follows that the Wald statistic below converges to a $\chi^2$-distributed random variable,  $$n \| \mathbb{I}(P_r)^{1/2} (\hat P_r - P_r) \|_F^2 \overset{d}{\rightarrow} \chi^2_{m_r(p-m_r)},$$
where we denote
\begin{equation}
\mathbb{I}(P_r)^{1/2}%=\frac{1}{\sqrt{2}} \sum_{s \neq r} (P_s \otimes P_r + P_r \otimes P_s) \frac{\lambda_r-\lambda_s}{\sqrt{\lambda_r \lambda_s}}
=\frac{1}{\sqrt{2}} \left ( \Sigma^{-1/2} C_r^{-1} \otimes P_r \Sigma^{-1/2} + \Sigma^{-1/2} P_r \otimes C_r^{-1} \Sigma^{-1/2} \right )
\end{equation}
and, slightly abusing notation, $C_r^{-1}:=\lambda_r I - \Sigma.$ \\
\noindent 
In the high-dimensional regime with $p \rightarrow \infty$ the test statistic above is stochastically unbounded and thus the $\chi^2$-approximation becomes invalid. Hence one has to further normalize, eventually obtaining a Gaussian limit instead. Moreover, higher order terms do not simply vanish anymore as $n \rightarrow \infty$. Particularly, when applying $\mathbb{I}(P_r)^{1/2}$ to $ ( \hat{P}_r-P_r)$ one has to multiply with $\Sigma^{-1/2}$ and thus the size of higher order error terms depends on the smallest eigenvalue of $\Sigma$ which we denote by $\lambda_{\min}$.
%The following Lemma shows that this intuition is justified in the growing dimension setting with $p \rightarrow \infty$, too, and quantifies the size of the error term.
\begin{lemma} \label{Lemma: Wald Statistic deterministic}
	Suppose that $\textbf{r}(\Sigma)=o(n)$ and that $\mathbb{E} \| \hat \Sigma-\Sigma \| \leq (1-\gamma)\bar g_r/2$ for some $\gamma \in (0,1)$.   %and that $\lambda_{\min} \geq c \mathbf{r}(\Sigma)/\sqrt{np}$. %Then, for a constant $C=C(\bar g_r, m_r, \| \Sigma \|, \lambda_r)$ and for every $t \geq 1$ we have with probability at least $1-e^{-t}$ that,
%	\begin{align}
%& \left |  \| \mathbb{I}_1^{1/2}(\hat P_r - P_r) \|_F^2 -  \|\mathbb{I}_1^{1/2} L_r\|_F^2- \mathbb{E} \left ( \| \mathbb{I}_1^{1/2} (\hat P_r - P_r) \|_F^2 -  \|\mathbb{I}_1^{1/2} L_r \|_F^2 \right ) \right |   \notag \\
%\leq & \frac{C(\bar g_r, m_r, \| \Sigma \|, \lambda_r)}{\lambda_{\min}}  \left ( \frac{\textbf{r}(\Sigma)}{n} \bigvee \frac{t}{n} \bigvee \frac{t^2}{n^2}\right ) \sqrt{\frac{t}{n}}.
%	\end{align}
For $g_i$ denoting $i.i.d. $ standard Gaussian random variables we have that, 
\begin{align} \frac{n  \| \mathbb{I}(P_r)^{1/2}(\hat P_r - P_r) \|_F^2-m_r(p-m_r)}{\sqrt{2m_r(p-m_r)}} \overset{d}{=} \frac{\sum_{i=1}^{m_r(p-m_r)} (g_i^2-1)}{\sqrt{2m_r(p-m_r)}} + Z_n,
\end{align}
where $Z_n$ fulfills with probability at least $1-e^{-t}$ for every $1 \leq t \leq n$ 
\begin{align} 
|Z_n| \leq & \notag C(\gamma, \bar g_r, m_r, \| \Sigma \|, \lambda_r )  \cdot \\ & \left (\frac{t}{\sqrt{n}} \bigvee \frac{(\textbf{r}(\Sigma) \vee t) \sqrt{t}}{\lambda_{\min}\sqrt{np}} \bigvee \sqrt{\frac{p}{n}}\sqrt{t} \bigvee \frac{t^{3/2}}{\sqrt{np}}  \bigvee \frac{\sqrt{p}\mathbf{r}(\Sigma)}{n} \bigvee \frac{\mathbf{r}(\Sigma )^2}{n\sqrt{p} \lambda_{\min}}\right ). \label{Errorterm}
\end{align}
\end{lemma}
%\textbf{Remark:} The constant $C'$ depends on $\| \Sigma \|$ and $\bar g_r$ and thus requires that these quantities are bounded, an assumption which is not given for example in the factor model type PCA described by \cite{WangFan}. \\ \\
%Note that in particular this implies that $n \mathbb{E} \| \mathbb{I}_1^{1/2} (\hat P_r - P_r) \|_F^2 = m_r(p-m_r)+ O_\mathbb{P}\left (  \frac{p}{n^{1/2}} \bigvee \frac{r(\Sigma)}{n^{3/2} \lambda_{\min} } \right )$. 
%Lemma \eqref{Lemma: Wald Statistic deterministic} can be used to prove a non-asymptotic Berry-Essen bound for the test-statistic $$\frac{n \left (\| \mathbb{I}^{1/2}_1(\hat P_r - P_r) \|_F^2 - \mathbb{E} \| \mathbb{I}^{1/2}_1(\hat P_r - P_r) \|_F^2 \right )}{\sqrt{2m_r(p-m_r)}}$$ as long as one imposes the assumption that $$\frac{\textbf{r}(\Sigma) \bigvee \log(n)}{\lambda_{\min}} \sqrt{\frac{\log(n)}{np}}=o(1).$$
%For example, in the spiked covariance model $\textbf{r}(\Sigma) \asymp p$ and $\lambda_{\min}$ is bounded from below and hence the condition above is fulfilled if $\textbf{r}(\Sigma)=o(n/\log(n))$. \\ 
%However, 
The bounds on $| \mathbb{E}\| \mathbb{I}(P_r)^{1/2}(\hat P_r - P_r)\|_F^2-m_r(p-m_r)|$  obtained in the proof of Lemma \ref{Lemma: Wald Statistic deterministic}  are sharp in their dependence on $p$ and $\mathbf{r}(\Sigma)$. Particularly this implies that without a further de-biasing step akin to the procedure in \cite{KoltchinskiiLouniciPCAarxiv} it is impossible to improve the dependence on $p$ and $\mathbf{r}(\Sigma)$ in \eqref{Errorterm}. \\ \\
In principle Lemma \ref{Lemma: Wald Statistic deterministic} could be used to construct confidence sets and tests for $P_r$. However, in statistical applications $\mathbb{I}(P_r)$ is usually not known and one needs to estimate it from the data. \\ 
For this we use the plug-in estimator given by %$\hat m_r$ and $\mathbb{I}_1^{1/2}$ by, 
\begin{align} \label{Def: Empirical Fisher Information}
\hat{\mathbb{I}}(\hat P_r)^{1/2}&:=%\frac{1}{\sqrt{2}} \sum_{s \neq r} \frac{\hat \lambda_r - \hat \lambda_s^{(n)}}{\sqrt{\hat \lambda_s^{(n)} \hat \lambda_r^{(n)}}} (\hat P_s\square \hat P_r^{(n)} + \hat P_r^{(n)} \square \hat P_s^{(n)})  \notag\\ & ~ =
\frac{1}{\sqrt{2}} \left [ \hat \Sigma^{-1/2} \hat C_r^{-1} \otimes  \hat P_r \hat \Sigma^{-1/2}+\hat \Sigma^{-1/2}\hat P_r\otimes \hat C_r^{-1} \hat \Sigma^{-1/2} \right ],
\end{align}
where we denote $\hat C_r^{-1}=\hat \lambda_r I - \hat \Sigma.$
To be able to consistenly estimate $\lambda_{\min }$ we need to assume that it is of larger order than $\| \hat \Sigma - \Sigma \| \asymp \sqrt{ {\bf r}(\Sigma)/n}$. 
Our main result is then following Berry-Essen type Theorem: \\
\begin{thm} \label{Thm fisher normalization}
	Suppose that $\textbf{r}(\Sigma)=o(n)$, that $\mathbb{E} \| \hat \Sigma-\Sigma \| \leq (1-\gamma)\bar g_r/2$ for some $\gamma \in (0,1)$, that $\bar g_r > \bar c$ for some constant $\bar c > 0$ large enough and that for another large enough constant $c>0$
	\begin{equation} \label{lambda min cond}
\lambda_{\min } \geq c \sqrt{\frac{{\bf{r}}(\Sigma) \bigvee \log(p)}{n}}.
	\end{equation} %$$ %\left (\frac{\textbf{r}(\Sigma)^2 \bigvee \log(n)^2}{n \sqrt{p}  \lambda_{\min}^3} \bigvee \textcolor{red}{\frac{ \textbf{r}(\Sigma)}{\sqrt{n \lambda_{\min}}}} \bigvee {\frac{\textbf{r}(\Sigma)}{\lambda_{\min}}} \sqrt{\frac{\log(n)}{{np} }} \right ) =o(1).  $$
	Then, for $\Phi$ denoting the distribution function of a standard Gaussian random variable we have that%we have for a constant $C$ depending on $\gamma,\bar g_r, m_r, \lambda_r$ and $ \| \Sigma \|$ that 
	\begin{align}
	& \sup_{x \in \mathbb{R}} \left | \mathbb{P} \left ( \frac{n \| \hat{\mathbb{I}}(\hat P_r)^{1/2} (\hat P_r- P_r )\|_F^2 -  m_r (p- m_r)}{\sqrt{2  m_r (p- m_r)}} \leq x  \right )-\Phi(x) \right | \notag \\ \leq & C(\gamma, \bar g_r, m_r, \| \Sigma \|, \lambda_r) \bigg [ \frac{1}{\sqrt{p}} +  \frac{\textbf{r}(\Sigma)^2 \vee \log(p)^2}{n \sqrt{p}  \lambda_{\min}^3} + \sqrt{\frac{p \log(p)}{n}}\notag \\ & + \frac{\left (\mathbf{r}(\Sigma) \vee \log(p) \right ) \sqrt{\log(p)}}{\lambda_{\min} \sqrt{np}} \notag  +
	 \frac{ \sqrt{p}\textbf{r}(\Sigma)}{n} %+\frac{\log(p)^{3/2}}{\sqrt{np}} 
	 \bigg ] .
	\end{align}
\end{thm}
% Thus, for $z_\alpha$ denoting the $1-\alpha$ quantile of a standard normal distribution and under the conditions of Theorem \ref{Thm fisher normalization} we obtain, 
% \begin{equation}
% \lim_{n, p \rightarrow \infty} \mathbb{P} \left ( (\| %n\hat{\mathbb{I}}_1^{1/2} (\hat P_r^{(n)} - P_r^{(n)} ) \|_F^2  \leq (1+M_n^{-1/2})( \hat m_r (p-\hat m_r)+z_\alpha \sqrt{2\hat m_r(p-\hat m_r)}) \right ) \geq 1-\alpha.
% \end{equation}
%
%
 % \textcolor{red}{possibly $r(\Sigma)=o(n^{2/3})$,}
%	Namely, for $\Sigma^{(n)}$ being a spiked covariance matrix fulfilling the assumptions of Theorem \ref{Thm fisher normalization} and assuming $r(\Sigma)=o(n^{1/2})$ we have, 
%$$\mathbb{P} \left ( n\| \hat{\mathbb{I}}_1^{1/2} (\hat P_r^{(n)} - P_r^{(n)} ) \|_F^2  \leq ( \hat m_r (p-\hat m_r)+z_\alpha \sqrt{2\hat m_r(p-\hat m_r)}) \right ) \overset{n,p \rightarrow \infty}{\longrightarrow} 1-\alpha.
% $$
 %Note that in case of the spiked covariance model with one spike, i.e. $\Sigma=\lambda s \otimes s + I$ the representation of $\mathbb{I}(P_r)^{1/2}$ simplifies to 
%$$\frac{\lambda}{\sqrt{2}(\lambda+1)} \left ( P_1^\perp \otimes P_1 + P_1 \otimes P_1^\perp \right ).$$
All quantities except $P_r$ in the Wald statistic above are known or, as in the case of $m_r$, can assumed to be known (see Lemma 2.2 in \cite{KoltchinskiiLofflerNickl}). This allows the construction of statistical tests and confidence ellipsoids for $P_r$. \\\\
Considering the spiked covariance model \eqref{Spiked covariance model} we have that
$\mathbf{r}(\Sigma)\asymp p$ and $\lambda_{\min} \asymp 1$ and thus one can meaningfully apply Theorem \ref{Thm fisher normalization} if $\mathbf{r}(\Sigma) =o(n^{2/3})$. \\  %Of course this does not include the ultra high-dimensional regime with $p/n \rightarrow const.$ (e.g. \cite{JohnstoneAos2001, PaulStatisticaSinica}). However, under this assumption results have been mostly negative (as in \cite{PaulStatisticaSinica}) and we are not aware of \emph{any} procedure enabling uncertainty quantification for $P_r$. 
% Thus, an interesting topic for further research would be the behaviour of the Wald statistic in this regime. While the pertubation results for $\hat P_r$ used here may be useful for this too, eventually random matrix theory might have to be employed. \\ \\
 %This %, together with the results by \cite{KoltchinskiiLofflerNickl},  
% resembles the findings by \cite{Portnoy1988} who proved that in exponential families the likelihood ratio test is consistent when $p=o(n^{2/3})$ whereas for asymptotic normality of single coordinates of the MLE $p=o(n^{1/2})$ may be needed. \\ \\
 If $\lambda_{\min}$ shrinks to $0$ the sample size requirements are becoming worse. For example, for models with decaying eigenvalues such that $\lambda_i \asymp i^{-\alpha}$, $0\leq  \alpha < 1$ we have that $\mathbf{r}(\Sigma) \asymp p^{1-\alpha}$ and $\lambda_{\min} \asymp p^{-\alpha}$. Therefore the application of Theorem \ref{Thm fisher normalization} is feasible if $p^{3/2 + \alpha}=o(n)$. \\ \\
In case of the spiked covariance model the conditions of Theorem \ref{Thm fisher normalization} compare favorably to the bootstrap approach used by \cite{NaumovSpokoinyUlyanov} who need to assume that $\textbf{r}(\Sigma) =o(n^{1/3})$. Moreover, for models with decaying eigenvalues with $\alpha < 3/8$ their condition is worse than our requirement $p^{3/2 + \alpha}=o(n)$ whereas for $\alpha > 3/8$ their condition is better.\\  The construction proposed in \cite{KoltchinskiiLouniciPCAAOS} and \cite{KoltchinskiiLouniciPCAarxiv} requires no assumption on  $\lambda_{\min}$ and allows for $\textbf{r}(\Sigma)=o(n)$, but instead relies on sample splitting into three samples, assumes that $m_r=1$ and that $\| \Sigma \|_F^2 \rightarrow \infty$. The last condition makes their Theorem unfeasible for application to models with quickly decaying eigenvalues $\lambda_i \asymp i^{-\alpha},~\alpha > 1/2$.   \\\\
%If $m_r=1$ using the identity
% $$ \frac{\| \hat P_r- P_r \|^2}{2}= \langle P_r-\hat P_r, P_r \rangle = \left \langle \left (P_r-\hat P_r \right ) \theta_r, \theta_r \right \rangle $$ 
% and assuming $r(\Sigma)=o(n^{1/2})$ Theorem \ref{Thm: Normalized limit bilinearform eigenvector}
% can in principle be used to construct asymptotic Frobenius confidence sets for  $P_r$. However, the resulting set of the form $\{ P_r \in \mathbb{H}: ~| \langle P_r-\hat P_r, P_r \rangle | \leq C \sqrt{n} \}$ would have  asymptotic coverage $1$ and Frobenius-diameter $n^{-1/4}$ which would be of greater order than the true risk $\sqrt{r(\Sigma)/n}$, rendering it sub-optimal for  practical applications. If $r(\Sigma) \gg n^{1/2}$ the identity above can not be used anymore as $\bar \theta_r$ does not have Euclidean norm $1$. \\
%\textbf{Remark:} Assuming $m_r=1$ one can use confidence sets for $ P_r$ to obtain confidence sets for $\theta_r$, using the elementary fact that $\langle P_r \theta_r, \theta_r \rangle=1$ for $\theta_r \in \mathbb{H}$ with $\|\theta_r \|=1$ if and only if $P_r = \theta_r  \theta_r^T$. If $C_n(P_r)$ is a confidence set for $P_r$ with coverage $1-\alpha$ then $\{ \theta_r:~\| \theta_r \|=1,  ~\langle P_r \theta_r, \theta_r \rangle=1, ~~P_r \in C_n(P_r) \}$ has coverage $1-\alpha$ for $\theta_r$.

\section{Proofs}
 %In the following we will assume that $\hat \lambda_r$ and $\hat P_r$ are the estimators obtained through this procedure.
 We first collect a few results from \cite{KoltchinskiiLouniciPCAAHP} and \cite{KoltchinskiiLouniciPCABernoulli} which we will frequently use throughout our proof. The first Lemma is a perturbation bound for spectral projectors proven in \cite{KoltchinskiiLouniciPCAAHP}. 
\begin{lemma} \label{Lemma: pertubation from K&L} Suppose that $\tilde \Sigma = \Sigma + E$. Let $\tilde P_r$ be the spectral projector corresponding to the eigenvalues $\{ \mu_j(\tilde  \Sigma), ~j \in \Delta_r \}.$ Then the following bound holds
	\begin{equation} \label{PertubationProjector1}
	\| \tilde P_r - P_r \| \leq 4 \frac{\| E \|}{\bar g_r}.
	\end{equation}
	Moreover,
	\begin{equation} \label{PertubationProjectorZerlegung}
	\tilde P_r - P_r = L_r(E)+S_r(E)
	\end{equation}
	where $L_r(E)=C_r E P_r + P_r E C_r$ and where $C_r$ denotes the resolvent operator
	\begin{equation} \label{Definition: C_r}
	C_r = \sum_{s \neq r } \frac{P_s}{\lambda_r - \lambda_s}
	\end{equation} and where the remainder term can be bounded 
	\begin{equation} \label{Bound: nonlinear S_r} \| S_r(E)\| \leq 14 \left ( \frac{\| E \|}{\bar{g}_r} \right )^2.
	\end{equation} \end{lemma}
In the course of our proofs we will also need a finer analysis of the non-linear term $S_r(E)$. 
\begin{lemma} \label{lemma higher order pertubation}
The following bound holds
	\begin{equation}
	S_r(E)=Z_r(E)+R_r(E)
	\end{equation}
	where \begin{equation}
	Z_r(E)=P_rEC_rEC_r+C_rEC_rEP_r+C_rEP_rEC_r-P_rEP_rEC_r^2-P_rEC_r^2EP_r-C_r^2EP_rEP_r
	\end{equation}
	and
	where the third order remainder term is symmetric and fulfills
	\begin{equation}
	\| R_r(E) \| \leq 72 \left ( \frac{\| E \|}{\bar g_r} \right )^3.
	\end{equation}
\end{lemma}
\begin{proof}
	The first part and symmetry of $R_r(E)$ follows immediately by inspecting the proof of Lemma 5 in \cite{KoltchinskiiLouniciPCAarxiv}. Moreover, 
 $$R_r(E)=-\frac{1}{2 \pi i} \oint_{\gamma_r} \sum_{k \geq3} (-1)^k (R_\Sigma(\eta)E)^k R_\Sigma(\eta)d \eta,$$
		where $\gamma_r$ denotes the circle of radius $\bar g_r/2$ centred at $\lambda_r$ with counterclockwise orientation and
		$R_\Sigma(\eta)$ denotes the resolvent of $P_r$, i.e.
		$$R_\Sigma(\eta)=\sum_{j \geq 1} \frac{P_j}{\mu_j-\eta}.$$
		Assume first that $\|E \| \leq \bar g_r/3$. Then we have that
		\begin{align}
		\left \| -\frac{1}{2 \pi i} \oint_{\gamma_r} \sum_{k \geq3} (-1)^k (R_\Sigma(\eta)E)^k R_\Sigma(\eta)d \eta \right \| & \leq 2 \pi \frac{\bar g_r}{2} \left (\frac{2}{\bar g_r} \right )^4 \| E \|^3 \sum_{k=0}^\infty  \left (\frac{2 \|E \|}{\bar g_r} \right )^k \notag \\
		& \leq 24 \left ( \frac{\| E \|}{\bar g_r}\right )^3 \notag.
		\end{align}
		If $\| E \|\geq \bar g_r/3$ then by Lemma \ref{Lemma: pertubation from K&L} and the explicit representation of the second order pertubation term in Lemma 5 in \cite{KoltchinskiiLouniciPCAarxiv} we obtain that
		\begin{align}
		& \left \| -\frac{1}{2 \pi i} \oint_{\gamma_r} \sum_{k \geq3} (-1)^k (R_\Sigma(\eta)E)^k R_\Sigma(\eta)d \eta \right \| \notag \\ \leq & \| \hat P_r - P_r \| + \| L_r(E)\| + \left \| -\frac{1}{2 \pi i} \oint_{\gamma_r} (R_\Sigma(\eta)E)^2 R_\Sigma(\eta)d \eta \right \| \notag \\
		\leq & 4 \frac{\| E \|}{\bar g_r} + \frac{2 \| E \|}{\bar g_r} + \frac{6 \| E \|^2}{\bar g_r} 
		\leq 72 \left (\frac{\|E\|}{\bar g_r} \right )^3 \notag
		\end{align}
\end{proof}
To bound $ \| \hat \Sigma - \Sigma \|$ we will frequently use the following bound and concentration inequality obtained by Koltchinskii and Lounici  in \cite{KoltchinskiiLouniciPCABernoulli}.
\begin{thm}  \label{Theorem bound Sigma from KL} Let $X_1,...,X_n$ be i.i.d. centred Gaussian random vectors with covariance matrix $\Sigma$ and such that $\mathbb{E}\|X_1 \|^2 < \infty$. Suppose that $\textbf{r}(\Sigma)=o(n)$.  Then, for some constant $C_q > 0$
	\begin{equation}
	\left (\mathbb{E} \| \hat \Sigma - \Sigma \|^q \right )^{1/q} \leq C_q \| \Sigma \| \sqrt{\frac{\textbf{r}(\Sigma)}{n}}.
	\end{equation}
	Moreover, there exists another constant $C' >0$ such that for all $t\geq 1$ with probability at least $1-e^{-t}$ we have that, 
	\begin{equation} \label{Bound: Concentration for Sigma}
	\left |  \| \hat \Sigma - \Sigma \| - \mathbb{E} \| \hat \Sigma - \Sigma \| \right |  \leq C' \| \Sigma \|  \left (  \sqrt{\frac{t}{n}} \bigvee \frac{t}{n}\right ).
	\end{equation}	
\end{thm}
%	In view of Lemma \ref{Lemma: Wald Statistic deterministic} it %suffices to obtain an expression for $n \mathbb{E} \| \mathbb{I}_1^{1/2} (\hat P_r - P_r) \|_F^2$, to show that $\hat m_r=m_r$ with probability going to $1$ and to bound $\| (\hat{\mathbb{I}}_1^{1/2}-\mathbb{I}_n^{1/2}) (\hat P_r - P_r) \|_F^2$. \\
In the following we denote $E=\hat \Sigma - \Sigma$, $L_r:=L_r(E)$, $S_r:=S_r(E)$, $Z_r:=Z_r(E)$ and $R_r:=R_r(E)$.
We now turn to the proof of Lemma \ref{Lemma: Wald Statistic deterministic}.
\begin{proof}[Proof of Lemma \ref{Lemma: Wald Statistic deterministic}]
		Going line by line through the proofs of Lemma 5, Theorem 5 and the calculations leading to display (5.17) in \cite{KoltchinskiiLouniciPCAAOS} it is easy to see that one can adjust them to show
	\begin{align} \frac{n  \left [ \| \mathbb{I}(P_r)^{1/2}(\hat P_r - P_r) \|_F^2-\mathbb{E} \| \mathbb{I}(P_r)^{1/2}(\hat P_r - P_r) \|_F^2  \right ]}{\sqrt{2m_r(p-m_r)}} \overset{d}{=} \frac{\sum_{i=1}^{m_r(p-m_r)} (g_i^2-1)}{\sqrt{2m_r(p-m_r)}} + Z_n',
	\end{align}
	where $Z_n'$ fulfills with probability at least $1-e^{-t}$ for every $1 \leq t \leq n$ 
	\begin{equation}
	|Z_n'| \leq C'(\gamma, \bar g_r, m_r, \| \Sigma \|, \lambda_r ) \left (\frac{t}{\sqrt{n}} \bigvee \frac{(\textbf{r}(\Sigma) \vee t) \sqrt{t}}{\lambda_{\min}\sqrt{np}} \bigvee \sqrt{\frac{p}{n}}\sqrt{t} \bigvee \frac{t^{3/2}}{\sqrt{np}} \right ).
	\end{equation}
	Thus it remains to obtain a tight bound for  $ \mathbb{E} \| \mathbb{I}(P_r)^{1/2} (\hat P_r - P_r) \|_F^2$. Using decomposition \eqref{PertubationProjectorZerlegung} we obtain that 
	\begin{equation} \label{Decomp expect of I} \mathbb{E} \| \mathbb{I}(P_r)^{1/2} (\hat P_r - P_r) \|_F^2 = \| \mathbb{I}(P_r)^{1/2} L_r \|_F^2 + \| \mathbb{I}(P_r)^{1/2} S_r \|_F^2 + \trace \left  ((\mathbb{I}(P_r)^{1/2} L_r)(\mathbb{I}(P_r)^{1/2} S_r) \right ). \end{equation}
	%We denote $E^{(n)}=\hat \Sigma^{(n)}-\Sigma^{(n)}$. 
As in the proof of Theorem 5 in \cite{KoltchinskiiLouniciPCAAOS} we obtain that 
	$$n \mathbb{E} \| \mathbb{I}(P_r)^{1/2} L_r \|_F^2= n \mathbb{E} \| P_r^\perp (\Sigma)^{-1/2} E (\Sigma)^{-1/2} P_r \|_F^2 =  m_r (p-m_r).$$
	Moreover, the second term in \eqref{Decomp expect of I} can be bound by applying Lemma \ref{Lemma: pertubation from K&L} and Theorem \ref{Theorem bound Sigma from KL} as follows \begin{align}
	\mathbb{E}  \| \mathbb{I}(P_r)^{1/2} S_r \|_F^2 \leq 2m_r \mathbb{E} \left \| C_r^{-1} \Sigma^{-1/2} S_r P_r \Sigma^{-1/2} \right \| \lesssim \frac{2  m_r \| \Sigma \|^2}{\bar g_r^4{\lambda_r \lambda_{\min}}} \mathbb{E} \| E \|^4= O \left ( \frac{\textbf{r}(\Sigma)^2}{n^2 \lambda_{\min} }\right ). \notag
	\end{align}
	For the last remainder term in \eqref{Decomp expect of I} the naive use of Cauchy-Schwarz does not suffice and and we have to use higher-order pertubation theory to obtain good enough bounds. %For given $k \in \mathbb{N}$ let $L$ be a non-empty strict subset of $\{1, \dots, k+1\}$ and denote by $\mathcal{L}_k$ the set of all such $L$. Define $m_L=|L|$ and denote by $V_L$ the set of vectors $\nu=(\nu_l: l \in L^c)$ with nonnegative integer components such that $\sum_{l \in L^c} \nu_l=m_L-1$. Then, by Lemma 5 from \cite{KoltchinskiiLouniciPCAarxiv} we have that
	%	\begin{equation} \label{Rep nonlinear term}
	%	S_r=\sum_{k\geq 2} \sum_{L \in \mathcal{L}_k}(-1)^{m_L-1} \sum_{\nu \in V_L}A_\nu(E)
	%	\end{equation}
	%	where $A_\nu(E):=B_1E \dots B_kEB_{k+1}$ with $B_l=P_r$ for $l \in L$ and $B_l=C_r^{\nu_l+1}$ for $l \in L^c$. 
	%Applying lemma \ref{lemma higher order pertubation} we have that 
	%	\begin{align}
	%	&%\left \langle \mathbb{I}_1^{1/2} L_r,
	%	 \sum_{L \in \mathcal{L}_2} (-1)^{m_L-1} \sum_{\nu \in V_L} \mathbb{I}_1^{1/2} A_\nu(E) %\right \rangle 
	%	 \notag \\
	%\mathbb{I}^{1/2} S_r= \frac{1}{\sqrt{2}} \Sigma^{-1/2} \left ( P_rEC_rEP_r^\perp + P_r^\perp E C_r E P_r - P_r E P_r EC_r - C_r E P_r E P_r\right ) \Sigma^{1/2} + \mathbb{I}^{1/2}(R_r(E)). \notag
	%	\end{align}
	Applying Lemma \ref{lemma higher order pertubation} and using symmetry of $L_r$, $R_r$, $Z_r$ and $\mathbb{I}(P_r)^{1/2}$ we obtain that,
	\begin{align} \label{Bound cross term I}
	&  \mathbb{E} \left \langle \mathbb{I}(P_r)^{1/2} L_r, \mathbb{I}(P_r)^{1/2} S_r  \right \rangle = \mathbb{E} \left \langle \Sigma^{-1} P_r E P_r^\perp \Sigma^{-1},P_r E C_r  EP_r^\perp - P_rEP_rEC_r \right \rangle \notag \\ & +   \mathbb{E} \left \langle \Sigma^{-1} P_r E P_r^\perp \Sigma^{-1},  P_rR_r(E)C_r^{-1} \right \rangle.
	\end{align}
	We now bound the three terms in \eqref{Bound cross term I} separately. 
	% Denoting $\Delta_r:=\{j: \mu_j = 
	For $\{ \theta_j\}_{j \in \Delta_r}$ denoting the eigenvectors of an eigen-decomposition of $P_r$ $\langle X_1, \theta_j \rangle$ and $P_r^\perp X_1$ are independent. Thus, we obtain that the first term in \eqref{Bound cross term I} equals
	\begin{align}
	& \mathbb{E} \left \langle \Sigma^{-1} P_r E P_r^\perp \Sigma^{-1}, P_r EC_rE P_r^\perp \right \rangle \notag \\  =  & \frac{1}{n^2} \sum_{j \in \Delta_r} \trace \left ( \Sigma^{-1} P_r^\perp X_1 X_1^T \theta_j \theta_j^T \Sigma^{-1} \theta_j \theta_j^T X_1  X_1^T C_r( X_1  X_1^T - \Sigma) P_r^\perp)\right ) \notag \\
	%= &\frac{1}{n^2 \lambda_r} \sum_{j \in \Delta_r} \mathbb{E} \trace (\Sigma^{-1} P_r^\perp (X_1 \otimes X_1) C_r ( X_1 \otimes X_1 - \Sigma) P_r^\perp ) \mathbb{E} \langle X_1, \theta_j \rangle^2 \notag \\
	= &\frac{m_r}{n^2} \sum_{i \neq l, i,l \notin \Delta_r} \frac{1}{\lambda_i(\lambda_r-\lambda_l)} \mathbb{E} \trace \left [ \theta_i  \theta_i^T X_1 X_1^T \theta_l \theta_l^T X_1  X_1^T \theta_i  \theta_i^T \right ]\notag  \\ \notag
	& + \frac{m_r}{n^2} \sum_{i \notin \Delta_r}  \frac{1}{\lambda_i(\lambda_r-\lambda_i)} \mathbb{E} \trace \left [ \theta_i \theta_i^T X_1  X_1^T \theta_i  \theta_i^T(X_1  X_1^T-\Sigma) \theta_i  \theta_i^T \right ] \notag \\
	= & \frac{m_r}{n^2} \left [\sum_{i \neq l, i,l \notin \Delta_r} \frac{1}{\lambda_i(\lambda_r-\lambda_l)} \mathbb{E} \langle X_1, \theta_i \rangle^2 \mathbb{E} \langle \theta_l, X_1\rangle^2 +  \sum_{i \notin \Delta_r} \frac{1}{\lambda_i(\lambda_r-\lambda_i)} ( \mathbb{E} \langle X_1, \theta_i \rangle^4 - \mathbb{E} \langle X_1, \theta_i \rangle^2 \lambda_i)\right ] \notag \\
	= & \frac{m_r}{n^2} \left [ \sum_{i \neq l, i,l \notin \Delta_r} \frac{\lambda_l}{\lambda_r-\lambda_l}+\sum_{i \neq \Delta_r} \frac{2 \lambda_i}{\lambda_r-\lambda_i} \right ] \lesssim \frac{m_r}{\bar g_r} \frac{p \mathbf{r}(\Sigma)}{ n^2}
	\end{align}

	The second term in \eqref{Bound cross term I} can be estimated similarly, 
	\begin{align}
	& \frac{1}{n^2\lambda_r } \sum_{j \in \Delta_r}  \mathbb{E} \trace \left ( \Sigma^{-1} P_r^\perp X_1  X_1^T \theta_j  \theta_j^T (X_1  X_1^T - \Sigma) \theta_j \theta_j^T X_1 X_1^T C_r  \right ) \notag \\
	= & \frac{2m_r \lambda_r}{n^2} \mathbb{E} \trace ( P_r^\perp \Sigma^{-1}X_1  X_1^T C_r) \notag \\
	= & \frac{2m_r \lambda_r}{n^2} \sum_{i \notin \Delta_r} \frac{\mathbb{E}\langle \theta_i, X_1\rangle^2}{\lambda_i(\lambda_r-\lambda_i)} \lesssim \frac{ m_r\| \Sigma \|}{\bar g_r} \frac{p}{n^2}.
	\end{align}
	The last term can be bound using Cauchy Schwarz, Lemma \ref{lemma higher order pertubation} and \ref{Theorem bound Sigma from KL},
	\begin{align} \label{Bound remainder}
	& \mathbb{E} \left \langle \Sigma^{-1} P_r E P_r^\perp \Sigma^{-1/2},  P_rR_r(E)C_r^{-1} \right \rangle  \notag \\\leq & \frac{\|\Sigma\| \sqrt{m_r}}{\sqrt{\lambda_r \lambda_{\min}}} \sqrt{\mathbb{E} \| P_r\Sigma^{-1/2} E \Sigma^{-1/2}P_r^{\perp} \|_F^2} \sqrt{\mathbb{E} \| R_r(E)\|^2} \notag \\
	\lesssim  & \frac{\| \Sigma \| m_r}{\bar g_r^3 \sqrt{\lambda_r \lambda_{\min}}} \sqrt{\frac{p}{n}} (\mathbb{E} \|E \|^6)^{1/2} \lesssim \frac{\| \Sigma \|^4 m_r}{\bar g_r^3 \sqrt{\lambda_r}} \frac{\sqrt{p} \mathbf{r}(\Sigma)^{3/2}}{n^2 \sqrt{\lambda_{\min}}}
	\end{align}
	Thus, summarizing, we have that $$\frac{n \mathbb{E} \| \mathbb{I}(P_r)^{-1/2}(\hat P_r - P_r) \|_F^2}{\sqrt{2 m_r(p-m_r)}} = \frac{m_r(p-m_r)}{\sqrt{2 m_r(p-m_r)}}+O \left ( \frac{\sqrt{p}\mathbf{r}(\Sigma)}{n}+\frac{\mathbf{r}(\Sigma )^2}{n\sqrt{p} \lambda_{\min}} \right ),$$ and the claim follows.
\end{proof}
We now turn to the proof of Theorem \ref{Thm fisher normalization} and show first that we can replace $\mathbb{I}(P_r)^{1/2}$ by $\hat {\mathbb{I}}(\hat P_r)^{1/2}$.
\begin{proof}[Proof of Theorem \ref{Thm fisher normalization}]
We have that
\begin{align}
 & \| (\mathbb{I}(P_r)^{1/2}-\hat{\mathbb{I}}(\hat P_r)^{1/2})(\hat P_r - P_r ) \|_F  \notag \\ \leq & 2\sqrt{ m_r}\| \hat P_r - P_r \|  \| \hat C_r^{-1} \| \| \hat \Sigma^{-1/2} \| \left ( \left |\frac{1}{\sqrt{\hat \lambda_r}}-\frac{1}{\sqrt{\lambda_r}} \right | + \frac{ \| \hat P_r - P_r \|}{\sqrt{\lambda_r}} \right ) \notag \\
 & + \frac{2\sqrt{ m_r }}{\sqrt{\lambda_r}}\| \hat P_r - P_r \| \bigg (\| \hat C_r^{-1}\| \| \hat \Sigma^{-1/2} - \Sigma^{-1/2} \| + \| C_r^{-1} - \hat C_r^{-1} \| \| \Sigma^{-1/2}\| \bigg) \notag \\
 =: & I+II+III+IV.
\end{align}
We now bound each of these four terms separately. We have that

\begin{align}
%& n \left \| \left (\hat \Sigma^{-1/2}\hat C_r^{-1} L_r \hat P_r \hat \Sigma^{-1/2} \right ) - \left (\hat \Sigma^{-1/2}\hat C_r^{-1} L_r  P_r  \Sigma^{-1/2} \right )  \right \|_F^2 \notag \\
%\leq & 2nm_r  \left \| \hat \Sigma^{-1/2}\hat C_r^{-1} \right \|^2 \| L_r \|^2 \left  \| \frac{\hat P_r}{\hat \lambda_r}  - \frac{ P_r}{ \lambda_r}  \right \|^2 \notag \\
%\leq & 4nm_r \frac{\| \hat \Sigma \|^2}{\hat \lambda_{\min}} \frac{\| E\|^2}{\bar g_r^2} \left ( \left  \| \frac{\hat P_r}{\hat \lambda_r}  - \frac{ P_r}{ \hat \lambda_r}  \right \|^2 + \left  \| \frac{ P_r}{\hat \lambda_r}  - \frac{ P_r}{ \lambda_r}  \right \|^2 \right ) \notag \\
%\leq & 32 n m_r \frac{\| \Sigma \|^2 \| E \|}{\lambda_{\min} \bar g_r^2} \left ( \frac{16 \| E \|^2}{\bar g_r^2 \hat \lambda_r} + \left (\frac{1}{\lambda_r}-\frac{1}{\hat \lambda_r}\right )^2 \right ) \notag \\
I \leq &  8\sqrt{m_r } \frac{\| E \|}{\bar g_r} \| \Sigma + E \| |\hat \lambda_{\min}|^{-1/2}   \sum_{k=1}^\infty \frac{| \hat \lambda_r - \lambda_r|^k}{\lambda_r^{k+1/2}}  \notag \\
 \leq & 16\sqrt{  m_r }\frac{\| E \|}{\bar g_r} \|\Sigma\| \left |\lambda_{\min}-\|E\| \right |^{-1/2}\sum_{k=1}^\infty \frac{\|E \|^k}{\lambda_r^{k+1/2}}
 \lesssim \frac{\sqrt{m_r} \| \Sigma \|^3}{\bar g_r \sqrt{\lambda_r}}\frac{\mathbf{r}(\Sigma) \vee t}{n \sqrt{\lambda_{\min}}},
\end{align}
with probability at least $1-e^{-t}$ for $1 \leq t \leq \log(p)$ and 
where we used  Theorem \ref{Theorem bound Sigma from KL}  to bound $\| E \|$,  Lidski's inequality to bound $| \hat \lambda_r- \lambda_r|$ and the $\lambda_{\min}$ condition \eqref{lambda min cond} to bound $  \lambda_{\min} - \| E\|  \geq \lambda_{\min}/2$. 
The second term can be bounded likewise, i.e. on the same event as the bound for $I$ we have with probability at least $1-e^{-t}$ for $1 \leq t \leq \log(p)$ that
\begin{align}
II & \leq \frac{32 \sqrt{ m_r + m_r }}{\sqrt{\lambda_r}} \frac{\|E\|^2}{\bar g_r^2} \| \hat \Sigma\| \| \hat \Sigma^{-1/2} \| \lesssim \frac{\sqrt{m_r} \| \Sigma \|^3}{\bar g_r^2 \sqrt{\lambda_r}}\frac{\mathbf{r}(\Sigma) \vee t}{n \sqrt{\lambda_{\min}}}.
\end{align}
%Thus, by the triangle inequality and adding and subtracting $\hat \Sigma^{-1/2}(\hat C_r)^{-1} L_r P_r \Sigma^{-1/2}$ we can bound \eqref{Bound Fisher info norm ii} by
%	\begin{align}
%	n \left \| \left ( \hat \Sigma^{-1/2}\hat C_r^{-1} - \Sigma^{-1/2} C_r^{-1}\right ) L_r \frac{P_r}{\sqrt{\lambda_r}} \right \|_F^2 +C'(\lambda_r, m_r, \bar g_r, \| \Sigma \|^2 )\left ( \frac{\textbf{r}(\Sigma)^2 \bigvee \frac{t^2}{n^2}}{n \lambda_{\min}} \right ). \label{Bound Fisher info norm ii 2}
%	\end{align}
%	Then, we can bound \eqref{Bound Fisher info norm ii 2} with probability at least $1-e^{-t}$ by
%	\begin{align}
%	& n \Bigg( \left \| \left ( \hat \Sigma^{-1/2} - \Sigma^{-1/2} \right )C_r^{-1}L_r \frac{P_r}{\sqrt{\lambda_r}}\right \|_F^2 \notag \\ + &  \left \| (\hat \Sigma)^{-1/2}\left ( \hat C_r^{-1} - C_r^{-1}\right )L_r \frac{P_r}{\sqrt{\lambda_r}} \right \|_F^2 \Bigg ) + C'(\lambda_r, m_r, \bar g_r, \| \Sigma \|^2 )\left ( \frac{\textbf{r}(\Sigma)^2 \bigvee t^2 }{n \lambda_{\min}} \right )
%	\label{Bound Fisher info norm ii 3}.
%	\end{align}
Using matrix series expansions of $ \Sigma \hat \Sigma^{-1}$ and of $(\Sigma \hat \Sigma^{-1})^{1/2}$ around $I$ we can bound the third term on the same event. We have with probability at least $1-e^{-t}$  for $1 \leq t \leq \log(p)$:
	\begin{align}
%	& m_r \left \| \left ( \hat \Sigma^{-1/2} - \Sigma^{-1/2} \right ) \right \|^2 \frac{\|  E \|^2}{\lambda_r}  \notag \\\leq & \frac{m_r \| \hat  E \|^2}{\lambda_r} \left \| \Sigma^{-1/2} \right \|^2 \left \| \left (I+\sum_{k=1}^\infty \Sigma^{-k}(\hat \Sigma - \Sigma^{(n})^k \right )^{1/2} - I\right \|^2 \notag \\
III  & \leq \frac{8\sqrt{ m_r}\| \hat \Sigma \| \|E \|}{\bar g_r \sqrt{\lambda_r}} \| \hat \Sigma^{-1/2} - \Sigma^{-1/2} \| \notag \\
	\lesssim  & \frac{\sqrt{ m_r} \| \hat \Sigma \| \|  E\|}{\bar g_r \sqrt{\lambda_{\min }\lambda_r}} \sum_{j=1}^\infty \left ( \sum_{k=1}^\infty \| \Sigma^{-1} \|^{k}\| E \|^{k} \right)^j \notag \\
	\lesssim & \frac{\sqrt{m_r} \| \Sigma\|^3}{\bar g_r\sqrt{\lambda_r}} \frac{\mathbf{r}(\Sigma) \vee t}{n \lambda_{\min}^{3/2}},
	\end{align}
	where we used again the $\lambda_{\min}$ condition \eqref{lambda min cond} to ensure convergence of the series. 
	Finally, the fourth term can be bound in the same fashion on the same event.  For $1 \leq t \leq \log(p)$ with probability at least $1-e^{-t}$ we have that
	\begin{align}
	IV & \leq \frac{8\sqrt{2 m_r} \| E \|}{\bar g_r \sqrt{\lambda_r \lambda_{\min}}} (|\hat \lambda_r - \lambda_r |+\| E \|) \lesssim \frac{m_r \|\Sigma \|^2}{\bar g_r \sqrt{\lambda_r}} \frac{\mathbf{r}(\Sigma) \vee t}{n \sqrt{\lambda_{\min}}}.
	\end{align}
	Thus, summarizing, and since we bounded $I,II,III$ and $IV$ on the same event we have, choosing $t=\log(p)$ with probability at least $1-1/p$ that
	\begin{align} \label{Finalbound}
	\frac{n\| (\mathbb{I}(P_r)^{1/2}-\hat {\mathbb{I}}(P_r)^{1/2}) (\hat P_r - P_r)\|_F^2}{\sqrt{2m_r(p-m_r)}} \leq C(\| \Sigma \|, \lambda_r, m_r, \bar g_r) \frac{\mathbf{r}(\Sigma)^2 \vee \log(p)^2}{n \sqrt{p} \lambda_{\min}^3}.
	\end{align}
Defining for random variables $\eta$ and $\xi$ $\Delta(\eta, \xi):=\sup_{x \in \mathbb{R}} \left | \mathbb{P} \left ( \xi \leq x \right ) -\mathbb{P} \left ( \eta \leq x \right ) \right |$ the anti-concentration bound in Lemma 4.6. from \cite{KoltchinskiiLofflerNickl} combined with \eqref{Finalbound} and Lemma \ref{Lemma: Wald Statistic deterministic}  thus implies that for $\eta:=\frac{n\| \hat {\mathbb{I}}(P_r)^{1/2} (\hat P_r - P_r)\|_F^2-m_r(p-m_r)}{\sqrt{2m_r(p-m_r)}}$
 and $Z$ denoting a standard Gaussian random variable we have that,
 \begin{align}
\Delta(\eta, Z) \leq & \Delta(\xi, Z) +  \frac{1}{p} \bigvee C(\gamma, \bar g_r, m_r, \| \Sigma \|, \lambda_r) \bigg [  \frac{\textbf{r}(\Sigma)^2 \vee \log(p)^2}{n \sqrt{p}  \lambda_{\min}^3} + \sqrt{\frac{p \log(p)}{n}}\notag \\ & + \frac{\left (\mathbf{r}(\Sigma) \vee \log(p) \right ) \sqrt{\log(p)}}{\lambda_{\min} \sqrt{np}} \notag  +
\frac{ \sqrt{p}\textbf{r}(\Sigma)}{n} \bigg ],
\end{align}
where the main term $\xi$ is defined as  $\sum_{i=1}^{m_r(p-m_r)}(g_i^2-1)/\sqrt{2m_r(p-m_r)}$.
Theorem \ref{Thm fisher normalization} now follows from the bound above and the Berry-Essen Theorem applied to $\Delta(\xi, Z)$.
\end{proof}

%\begin{supplement}
%\sname{Supplement A}\label{suppA}
%\stitle{Title of the Supplement A}
%\slink[url]{http://www.e-publications.org/ims/support/dowload/imsart-ims.zip}
%\sdescription{Dum esset rex in
%accubitu suo, nardus mea dedit odorem suavitatis. Quoniam confortavit
%seras portarum tuarum, benedixit filiis tuis in te. Qui posuit fines tuos}
%\end{supplement}

\end{document}